\documentclass[a4paper, reqno, 12pt]{amsart}
\usepackage[english]{babel}
\usepackage{amsmath,amssymb,amsthm, verbatim}
\usepackage{calrsfs}
\usepackage{dsfont}

\usepackage[pdftex]{xcolor}

\usepackage[bookmarks=true, hyperindex, pdftex, colorlinks, allcolors=blue]{hyperref}
\usepackage{mathtools}
\usepackage{enumerate}

\numberwithin{equation}{section}

\theoremstyle{plain}
\newtheorem{theorem}{Theorem}[section]
\newtheorem{corollary}[theorem]{Corollary}
\newtheorem{prop}[theorem]{Proposition}
\newtheorem{lemma}[theorem]{Lemma}

\theoremstyle{definition}

\newtheorem{remark}[theorem]{Remark}

\newtheorem{definition}[theorem]{Definition}

\newcommand{\R}{\mathbb{R}}
\newcommand{\N}{\mathbb{N}}

\newcommand{\rar}{\rightarrow}

\newcommand{\eins}{\mathds{1}}

\newcommand{\integral}[1]{\int\limits_0^1 {#1} {\rm dt}}

\newcommand{\eps}{\varepsilon}
\newcommand{\del}{\delta}
\newcommand{\Del}{\Delta}

\newcommand{\conv}{{\rm conv}}
\newcommand{\cconv}{{\overline{\conv}}}

\newcommand{\bc}[1]{{\rm clb}(#1)}
\newcommand{\bb}[1]{{\rm b}{(#1)}}
 \DeclareMathOperator{\diam}{{\rm diam}}
\newcommand{\hausd}{\rho_H}
\newcommand{\whausd}{\widetilde{\rho}_H}
 \DeclareMathOperator{\codim}{{\rm codim }}

\newcommand{\dt}{{\rm dt}}

\renewcommand{\leq}{\leqslant}
\renewcommand{\geq}{\geqslant}
\renewcommand{\le}{\leqslant}
\renewcommand{\ge}{\geqslant}

\usepackage{etoolbox}

\patchcmd{\abstract}{\titlepage}{\thispagestyle{empty}}{}{}
\patchcmd{\endabstract}{\endtitlepage}{\clearpage}{}{}

\begin{document}
\title[ Integrability of a multifunction and of its convex hull]{Connection between the Riemann integrability of a multi-valued function and of its convex hull}

\author[V.~Kadets]{Vladimir Kadets}

\address[Kadets]{School of Mathematics and Computer Science \\
 V.~N.~Karazin Kharkiv National University \\
Svobody square~4 \\
61022~Kharkiv \\ Ukraine
\newline
\href{http://orcid.org/0000-0002-5606-2679}{ORCID: \texttt{0000-0002-5606-2679} }
}
\email{v.kateds@karazin.ua}

\author[A.~Kulykov]{Artur Kulykov}

\address[Kulykov]{School of Mathematics and Computer Science \\
 V.~N.~Karazin Kharkiv National University \\
Svobody  square~4 \\
61022~Kharkiv \\ Ukraine
}
\email{artur.kulykov.04.02@gmail.com}

\author[O.~Shevchenko]{Olha Shevchenko}

\address[Shevchenko]{School of Mathematics and Computer Science \\
 V.~N.~Karazin Kharkiv National University \\
Svobody  square~4 \\
61022~Kharkiv \\ Ukraine
}
\email{shevchenkoos@yahoo.com}

\thanks{ The research of the first and the third authors was supported by the National Research Foundation of Ukraine funded by Ukrainian State budget as part of the project 2020.02/0096 ``Operators in infinite-dimensional spaces: the interplay between geometry, algebra and topology''}

\begin{abstract}
For a Banach space $X$ we demonstrate the equivalence of the following two properties: 

(1) $X$ is B-convex (that is, possesses a nontrivial infratype), and 

(2) if ${F: [0,1] \to 2^{X} \setminus \{\varnothing\}}$ is a {multifunction}, $\conv F$ denotes the mapping $t \mapsto \conv F(t)$, then the Riemann integrability of $\conv F$ is equivalent to the Riemann integrability of $F$.

For multifunctions with compact values the Riemann integrability of $\conv F$ is equivalent to the Riemann integrability of $F$ without any restrictions on the Banach space $X$.
\end{abstract}

\date{}

\subjclass[]{}

\keywords{}

\maketitle
\thispagestyle{empty}

\section{Introduction}

The most popular definition of integrability for functions with values in a Banach space is the Bochner integrability, which is a direct generalization of the ordinary Lebesgue integration, see a very nice short exposition in \cite[Chapter 5]{BenLind}. Bochner integrability assumes measurability, separability of the set of values outside of values on a set of measure 0, and Lebesgue integrability of the real function $t \mapsto \|f(t)\|$. This type of integration has very nice properties, but is restrictive. In particular, some Riemann integrable functions are not integrable in Bochner sense. Let us recall the definition and the corresponding well-known illustrative example (\cite{Kadets1988}, \cite[Pages 119-120]{Kadets}).

\begin{definition}
 Let $T = \{\Delta_i, t_i \}_{i=1}^n$ be a tagged partition of $[0,1]$ ($t_i \in \Delta_i$). Denote $d(T) = \max\limits_{1\leq i \leq n} |\Delta_i|$. For a bounded function $f\colon [0;1]\rar X$ define the corresponding \textit{Riemann} sum as $S(f,T) = \sum \limits_{i=1}^n |\Delta_i| F(t_i)$. The function $f$ is said to be {\emph{Riemann integrable}} if there exists an $x\in X$ (called $f$'s \emph{Riemann integral}) such that for any $\eps >0$ there is a $\del >0$ such that for any tagged partition $T$ with $d(T) <\del$ the corresponding Riemann integral sum is $\eps$-close to $x$, that is, $$\| S(f, T)-x\| <\eps. $$
 In other words, $x \in X$ is {integral} of $f$ ($x = \int\limits_0^1 f(t) \dt$) if $x = \lim \limits_{d(T) \to 0} S(f,T)$.
\end{definition}

Recall that the space $\ell_2([0;1])$ consists of all functions $x\colon [0;1]\rar \R$ with\\ $\sum_{t\in [0;1]}|x(t)|^2<\infty $ (in particular supports of these functions are at most countable subsets of $[0;1]$). The norm on $\ell_2([0;1])$ is given by $\|x\|=(\sum_{t\in [0;1]}|x(t)|^2)^{1/2}$. It is easy to see that $\ell_2([0;1])$ is a non-separable Hilbert space. The standard orthonormal basis consists of continuum number of elements $e_t :=\eins _{\{t\}}$, $t \in [0, 1]$.

Now consider the function $f\colon [0;1]\rar \ell_2([0;1])$: $f(t)=e_t$. This function is not measurable: if $A$ is a non-measurable subset of $[0;1]$, then $A$ is the preimage of $\{e_t: t \in A\}$, but the later is a closed set and hence a Borel subset of $\ell_2([0;1])$. 

However, this function is Riemann integrable. To see this, note that for any tagged partition partition $T$, if $d(T) \le \eps$ then
\begin{eqnarray*}
\|S(f,\Pi ,T)\| &=&\left\|\sum f(t_i)\left| \Del_i\right| \right\|=
\left\|\sum e_{t_i}\left| \Del_i\right| \right\|= \left(\sum \left| \Del_i\right| ^2\right)^{1/2}\\
\ & \leq& \left(\sum \eps \cdot \left|
\Del_i\right| \right)^{1/2} = \sqrt{\eps}\Bigl(\sum_i\left| \Del_i\right| \Bigr)^{1/2}=\sqrt{\eps}.
\end{eqnarray*}
Thus $f$ is Riemann integrable with the zero integral.

In reality, outside of Bochner and Riemann integrals, there are a lot of other useful non-equivalent integrability concepts for functions with values in a Banach space: Pettis integral, Gelfand integral, Birkhoff integral, McShane integral, Henstock--Kurzweil integral, and others.
The relationship between various kinds of integrability is a huge and popular area of research.

A substantial part of modern integration theory deals with set-valued functions. The most popular approaches to set-valued integration belong to Nobel laureates in economics Aumann~\cite{aum} (through selectors) and Debreu~\cite{dbr} (through R{\aa}dstr\"om embedding which maps convex compact sets into points of the space of bounded functions on $B_{X^*}$, after which the multifunction transforms into an ordinary function and the ordinary Bochner integration is applicable). These approaches are proved to be very useful in various areas like optimization, differential inclusions and mathematical economics. Both approaches combined with different kinds of integrals for ordinary Banach space-valued functions lead to a number of integrals for set-valued functions, and, moreover, some vector-valued integrals like Pettis or Gelfand integral have direct generalizations to the set-valued case, that are not equivalent to those, obtained from Aumann's and Debreu's approaches (\cite{cas-kad-rod, cas-kad-rod-2, cas-kad-rod-3, cas-rod-2} and citations therein). For a detailed account on measurable selection results and set-valued integration we refer the reader to the monographs \cite{cas-val,kle-tho} and the survey~\cite{hes-J}.

In this paper we concentrate on the Riemann integration of set-valued functions. The letter $X$ is reserved for Banach spaces. We consider only real Banach spaces and use the standard Banach space theory notation like $X^*$ for the dual space, $L(X,Y)$ for the space of bounded linear operators from $X$ to $Y$, or $B_X$ for the closed unit ball without additional explanation. Below a \textit{set-valued function} (another name -- \textit{multifunction}) with values in $X$ is a map $F: [0, 1] \to 2^{X} \setminus \{\varnothing\}$.

If $X$ is a normed space, we can add elements of $2^X$ and multiply them by a constant in a natural way: for subsets $A, B \subset X$ their (Minkowski) sum is $A+B := \{a+b \colon a\in A, b\in B \}$, and if $\lambda$ is a scalar, $\lambda A := \{\lambda a \colon a\in A\}$.

Denote by $\bb X$ the collection of all not empty bounded subsets of $X$ and $\bc X$ the subcollection of all closed bounded subsets. The convergence in $\bb X$ will be considered with respect to the Hausdorff distance $\hausd$:
\begin{definition}
 The \textit{one-sided Hausdorff distance} between sets $A, B \in \bb X$ is the number $\whausd(A,B) = \sup_{b\in B} \rho(b,A)$. The \textit{Hausdorff distance} between $A$ and $B$ is the number 
$$
\hausd(A,B) = \max \{ \whausd(A,B), \whausd(B,A)\}.
$$ 
\end{definition}
Remark, that $\hausd$ is a pseudometric on $\bb X$ but is not a metric: $\hausd(A, \overline{A}) = 0$ for every $A \in \bb X$. This means that the limit in Hausdorff distance is not uniquely defined. In order to eliminate this inconvenience, below in the definition of the Riemann integral of a multifunction we demand that integral belongs to $\bc X$: since for two different closed sets the Hausdorff distance is non-zero, this makes the integral unique.
\begin{definition}
For a tagged partition $T = \{\Delta_i, t_i \}_{i=1}^n$ of $[0,1]$ and a multifunction $F: [0,1] \to \bb X$ define the corresponding \textit{Riemann sum} $S(F,T) = \sum \limits_{i=1}^n |\Delta_i| F(t_i)$. Then a set $A \in \bc X$ is said to be the \textit{Riemann integral} of $F$ ($A = \int\limits_0^1 F(t) \dt$) if $A = \lim \limits_{d(T) \to 0} S(F,T)$, where the limit is taken in the sense of the Hausdorff distance. If the Riemann integral exists, $F$ is called \textit{Riemann integrable}.
\end{definition} 

We were able to trace this natural definition of the multi-valued Riemann integral to \cite{Hukuhara} for multifunctions with convex compact values and to \cite{Polovin1974} without the convexity assumption.

\begin{remark}
Since in our paper we consider only Riemann integrability, below we use words ``integral'' and ``integrable'' in the sense of ``Riemann integral'' and ``Riemann integrable''.
\end{remark}

\begin{definition}
Let $F: [0,1] \to 2^X$ be a multifunction. Then the \textit{convex hull} of $F$ is the multifunction $\conv F: [0,1] \to 2^X: t \mapsto \conv(F(t))$.
\end{definition}

\begin{theorem} \label{conv_integr}
Let $F$ be an integrable multifunction with $\int\limits_0^1 F(t) \dt = A$ . Then its convex hull $\conv F$ is also integrable and $\integral{\conv F(t)} = \overline{\conv} A$.
\end{theorem}
\begin{proof}
Since the convex hull of Minkowski sum is the Minkowski sum of convex hulls, $S(\conv F, T) = \conv S(F,T)$ for every tagged partition $T$ of $[0,1]$. 
It remains to apply the evident inequality
\begin{equation*} \label{eq-ev-ineq}
\hausd(\conv U, \conv V) \leq \hausd(U, V)
\end{equation*}
which is true for all $U, V \in \bb X$:
$$
\hausd (S(\conv F,T),\cconv A) = \hausd (\conv S(F,T),\conv A) \leq \hausd (S(F,T), A) \xrightarrow[d(T)\to 0]{} 0.
$$
So $\integral{\conv F(t)} = \cconv A$. 
\end{proof}

The converse to Theorem \ref{conv_integr} result together with the stronger equality 
\begin{equation} \label{eq-int-conv-hull}
\integral{\conv F(t)} = \integral{F(t)}
\end{equation}
was demonstrated in \cite[Theorem 10]{Polovin1983} under additional assumption of finite-dimensionality of $X$. Later it was extended to functions with compact values under assumption that $X$ is superreflexive (see \cite[Corollary 3.1]{Iv-Pol2012} and the discussion around it).

The main results of our paper are:
\begin{enumerate}[{1.}]
 \item The integral of any Riemann integrable multifunction is convex.
 \item The converse to Theorem \ref{conv_integr} holds true in any Banach space $X$ with an infratype $p > 1$ (so-called B-convex spaces, see Definition \ref{def-infratype}).
 \item If a Banach space $X$ is not B-convex then there exists a multifunction $F: [0,1] \to \bb X$ which is not Riemann integrable, but whose convex hull is Riemann integrable.
 \item In any Banach space $X$ if the convex hull of multifunction $F: [0,1] \to \bb X$ is Riemann integrable with compact $\integral{\conv F(t)}$, then $F$ is Riemann integrable and \eqref{eq-int-conv-hull} holds true. In particular, this happens if all the values of $F$ are precompact.
\end{enumerate}

The structure of the paper is as follows. In the next section ``\nameref{sect2}'' we list elementary properties of integral and demonstrate the first (and simplest) of the aforementioned main results. After that, in the section ``\nameref{sect3}'' we give the necessary preliminaries from the Banach space theory and demonstrate the second and the third of the main results; and the last section ``\nameref{sect4}'' is devoted to the last result from the above list.

\section{The convexity of integral} \label{sect2}

Let us list without proof some easy-to check properties.
 
\begin{prop} \label{prop21++}
Let $F: [0,1] \to \bb X$ be a Riemann integrable multi-valued function with 
$\int\limits_0^1 F(t) {\dt} = A$ and $P \in L(X,Y)$. Then $P \circ F$ $($the multifunction $t \mapsto P(F(t))$$)$ is also integrable and $\integral{P(F(t))} = \overline{P(A)}$. 
\end{prop}
\begin{prop}
Every Riemann integrable multifunction is bounded. That is, if $F$ is integrable, then
$\sup\{\|a\| \colon a \in F(t), t \in [0,1]\} < \infty.$
\end{prop}
\begin{prop} \label{prop-const-int1}
If $A \in \bb{X}$ is convex and $F(t) = A$ for all $t \in [0,1]$, then $\integral{F} = \overline{A}.$
\end{prop}

\begin{theorem} \label{thm-conv}
The integral of any Riemann integrable multifunction is convex.
\end{theorem}

\begin{proof}
Let $F$ be a Riemann integrable multifunction, $A = \integral{F(t)}$. Since $A$ is closed, in order to demonstrate its convexity it is sufficient to prove that $\frac12 A + \frac12 A \subset A$.

Consider a sequence of partitions $ \Pi_n =\{\Delta_{k,n}\}_{k=1}^{N_n}$ with $\lim_{ n \to \infty}d(\Pi_n) = 0,$. Let $\Delta_{k,n} = [x_{(k-1),n}, x_{k,n}]$. 
Define a sequence of partitions $\widetilde \Pi_{n} = \{\widetilde{\Delta}_{i,n}\}_{i=1}^{2N_n}$, where 
$$
\widetilde{\Delta}_{2k-1, n} = \left[x_{(k-1),n}, \frac12 x_{(k-1),n} + \frac12 x_{k,n}\right], \, \widetilde{\Delta}_{2k,n} = \left[\frac12 x_{(k-1),n} + \frac12 x_{k,n}, x_{k,n}\right],
$$
that is we divide every segment of the partition $\Pi_n$ into two segments of equal length. 
Selecting arbitrary points $t_{i,n} \in \widetilde{\Delta}_{i,n}$ we define the sequence of tagged partitions $\widetilde{T}_n = \{\widetilde{\Delta}_{i,n}, t_{i,n}\}_{i=1}^{2N_n}$. Now the main trick comes: both points $ t_{2k-1,n}$ and $t_{2k,n}$ lie in the original bigger interval $\Delta_{k,n}$, so for a given $n \in \N$ we can define two tagged partitions based on the same selection of segments $\{\Delta_{k,n}\}_{k=1}^{N_n}$, namely $T_{a,n} = \{\Delta_{k,n}, t_{2k-1,n}\}_{k=1}^{N_n}$ and $T_{b,n} = \{\Delta_{k,n}, t_{2k,n}\}_{k=1}^{N_n}$. 
Then, for the corresponding Riemann sums 
$$S_{a,n} = S(F,T_{a,n}), S_{b,n} = S(F,T_{b,n}) \text{ and } S_n = S(F,T_n)$$
we have that $S_n = \frac12 S_{a,n} + \frac12 S_{b,n}$, so
\begin{equation} \label{eq-riem-sum12}
\rho_H\left(S_n, \frac12 S_{a,n} + \frac12 S_{b,n}\right) = 0.
\end{equation}
By the integrability, all the three limits $\lim_{ n \to \infty} S_n$, $\lim_{ n \to \infty} S_{a,n}$, and $\lim_{ n \to \infty} S_{b,n}$ exist and are equal to $A$, so passing in \eqref{eq-riem-sum12} to the limit we obtain that 
$$
\rho_H\left(A, \frac12 A + \frac12 A \right) = 0, 
$$
which means that $\overline{\frac12 A + \frac12 A} = A$, which gives the desired inclusion $\frac12 A + \frac12 A \subset A$. 
\end{proof}
Combining this proposition with Theorem \ref{conv_integr} we get:
\begin{corollary} \label{cor25u}
Suppose $F$ is a Riemann integrable multifunction. Then $\conv F$ is also integrable and $\integral{\conv F(t)} = \integral{F(t)}$.
\end{corollary}

\section{The connection between a function and its convex hull} \label{sect3}

\begin{definition}\label{def-infratype}
 A Banach space $X$ is said to be of infratype $p$ with a constant $C> 0$ if for every finite set $\{x_k\}_{k=1}^n \subset X$ the following inequality holds:
 $$\min\limits_{\alpha_i = \pm 1} \left\|\sum\limits_{i=1}^{n} \alpha_i x_i\right\| \leq C\left(\sum\limits_{i=1}^n \|x_i\|^p \right)^{1/p}.$$
\end{definition}
In the sequel, we say that $X$ is of infratype $p$ if there exists a $C>0$ such that $X$ is of infratype $p$ with the constant $C$. The theory of type and infratype was originated by Gilles Pisier \cite{Pisier} and is an important part of local theory of Banach spaces, see the survey \cite{Pisier2}, and books \cite[Sections II.E, III.A]{Wojt} and \cite[Chapter 5]{Kadets}. It is easy to see that

\begin{itemize}
 \item any Banach space is of infratype $1$ (for this reason the infratype 1 is called \emph{the trivial infratype});
\item infratype is an isomorphic invariant;
 \item any Hilbert space is of infratype $2$, consequently finite-dimensional spaces have infratype $2$ (in reality, finite dimensional spaces have a much stronger property \cite[Lemma 2.2.1]{Kadets});
 \item the space $\ell_1$ does not have a non-trivial infratype.
\end{itemize}

In this section we are going to prove the following theorem: 
\begin{theorem}
Let $X$ be a Banach space. Then the following assertions are equivalent: 
\begin{enumerate}[{$($1$ )$}]
 \item $X$ has an infratype $p > 1$;
 \item For every multifunction $F: [0,1] \to \bb{X}$ if $\conv F$ is integrable, then $F$ is integrable itself.
\end{enumerate}
\end{theorem}

The implications (1)$\Rightarrow$(2) and (2)$\Rightarrow$(1) in the above theorem are of very different nature, and will be proved as two separate theorems: Theorem \ref{thm-in-infratype} and Theorem \ref{thm-without-infratype} in two different subsections of this section.

\subsection{Spaces with a non-trivial infratype}
\begin{lemma}[{\cite[Lemma before Theorem 3]{Kadets1988}, or \cite[P. 133-134, Lemma 3]{Kadets}}] \label{lsum}
Let $p > 1$, $C > 0$, and $X$ be a Banach space with infratype $p$ and constant $C$. Let $\{A_i\}_{i=1}^{n}$ be an arbitrary collection of bounded subsets of $X$, $d_i = diam(A_i)$, and $\{b_i\}_{i=1}^{n}$ be a collection of points such that $b_i \in \conv A_i$, $i = 1, \ldots, n$. 
Then one can choose points $a_i \in A_i$, so that 
$$
\left\| \sum_{i=1}^{n} (a_i - b_i) \right\| \leq C_1 \left( \sum_{i=1}^{n} d_i^p \right)^{1/p},
$$
where $C_1 = \frac{2C}{2^{1-\frac{1}{p}} - 1}$.
\end{lemma}

\begin{lemma} \label{ldiam}
 Let $p>1$, $\{d_i\}_{i=1}^{n} \subset (0, +\infty)$ with $\sum_{i=1}^{n} d_i = 1$. Denote $d = \max_i d_i$. Then $\sum_{i=1}^{n} d_i^p \leq d^{p - 1}$.
\end{lemma}
\noindent\emph{Proof}.
Since $d_i\leq d$ and $p-1>0$, $d_i^{p-1} \leq d^{p-1}$. Then 
$$\sum_{i=1}^{n} d_i^p = \sum_{i=1}^{n} d_i^{p-1} \cdot d_i \leq \sum_{i=1}^{n} d^{p-1} \cdot d_i = d^{p-1} \cdot \sum_{i=1}^{n} d_i = d^{p-1}. \eqno{\Box}$$


\begin{theorem} \label{thm-in-infratype}
Let $X$ be a space with an infratype $p>1$ with a constant $C$, $F:[0;1] \to \bb{X}$ be Riemann integrable with $\integral{\conv F(t)}= A$. Then $\integral{F(t)} = A$.
\end{theorem}
\begin{proof}
Since $\conv F$ is integrable, it is bounded, hence, $F$ is also bounded. Denote $M = \sup\limits_{t} \diam(F(t))$. For a tagged partition $T = \{\Delta_i, t_i \}_{i=1}^n$ of $[0,1]$ consider an arbitrary $b \in S(\conv(F),T)$. It can be written in the form $b = \sum\limits_{i=1}^n b_i $ with $b_i \in \conv F(t_i) |\Delta_i|$. Applying Lemma \ref{lsum} to the collection $\{b_i\}_{i=1}^n$ we obtain $\{a_i\}_{i=1}^n$ such that $a_i \in F(t_i) |\Delta_i|$ and
\begin{equation*}
\begin{split}
 \left\| \sum_{i=1}^{n} a_i - \sum_{i=1}^{n} b_i \right\| =
 \left\| \sum_{i=1}^{n} (a_i - b_i) \right\| \leq 
 C_1 \left( \sum_{i=1}^{n} (|\Delta_i| \diam(F(t_i)))^p \right)^{1/p} \leq \\
 C_1 \left( \sum_{i=1}^{n} (M |\Delta_i|)^p \right)^{1/p} =
 C_1 M \left( \sum_{i=1}^{n} |\Delta_i|^p \right)^{1/p} \leq C_1 \cdot M \cdot d(T)^\frac{p-1}{p},
\end{split}
\end{equation*}
where $d(T) = \max\limits_i |\Delta_i|$ and $C_1 = \frac{2C}{2^{1-\frac{1}{p}} - 1}$ (on the last step we used Lemma \ref{ldiam} with $d_i = |\Delta_i|$).
So, for every $b \in S(\conv(F),T)$ there is $ a = \sum\limits_{i=1}^n a_i \in S(F,T)$ such that $\|a-b\| \leq C_1 \cdot M \cdot d(T)^\frac{p-1}{p}$ .
Consequently, 
$$
\whausd(S(F,T), S(\conv(F),T)) = \sup\limits_{b\in S(\conv(F),T)} \rho(b, S(F,T)) \xrightarrow[d(T) \rightarrow 0]{} 0.
$$
On the other hand, $S(F,T) \subset S(\conv(F), T)$, so $ \whausd(S(\conv(F),T), S(F,T)) = 0$. 
This means that $\hausd(S(F,T), S(\conv F, T)) \xrightarrow[d(T) \rightarrow 0]{} 0, $ and
$$ 
\integral{F(t)} = \lim \limits_{d(T) \to 0} S(F,T) = \lim \limits_{d(T) \to 0} S(\conv F, T) = A
$$
as desired. 
\end{proof}

\subsection{Spaces without non-trivial infratype}
\begin{definition}
 Let $X$ and $Y$ be Banach spaces. The\textit{ Banach-Mazur distance} between $X$ and $Y$ is the number $d(X, Y) = \inf\{\|T\| \cdot \|T^{-1}\|\}$, where the infinum is taken over all isomorphisms $T: X \to Y$. If $X$ and $Y$ are not isomorphic, we put $d(X, Y) = +\infty$. 
\end{definition}
\begin{definition}
 \label{fin_repr}
 A Banach space $Y$ is said to be \emph{finitely representable} in a space $X$ ($Y \xrightarrow{f} X$ for short) if for any $\eps> 0$ and any finite-dimensional subspace $Z \subset Y$ there exists a finite-dimensional subspace $Z_1 \subset X$ such that $$d(Z,Z_1) < 1 + \eps.$$ 
\end{definition} 
\begin{definition}
$X$ is said to be \emph{B-convex} if the space $\ell_1$ is not finitely representable in $X$. 
\end{definition}
\begin{theorem}[\cite{Pisier}]
 \label{finrepr}
 A space $X$ is B-convex if and only if it has an infratype $p>1$.
\end{theorem}
\begin{definition}
Let $\eps \in (0, 1)$.  A subspace $Z$ of a Banach space $X$ is said to be \emph{$\eps$-orthogonal} to a subspace $Y \subset X$ if for all $y \in Y$ and $z\in Z$
\begin{equation} \label{eqqq}
\|y+z\| \geq (1-\eps) \|y\|.
\end{equation}
\end{definition}

Remark, that  $Z$ and $Y$ in the above definition intersect only by the zero element of $X$. Indeed, if there is a non-zero $y \in Y \cap Z$, then $z := -y \in Z$ but $\|y+z\| = 0$, which contradicts \eqref{eqqq}.

\begin{lemma}[{\cite[Lemma 6.3.1]{Kadets}}] \label{perp_exists}
Let $X$ be a Banach space, $Y$ be a finite-dimensional subspace of $X$, and $Z$ be an infinite-dimensional subspace of $X$. Then for any $\eps \in (0, 1)$  there exists a subspace $G \subset Z$ of finite codimension in $Z$, such that $G$ is $\eps$-orthogonal to $Y$. 
\end{lemma}

\begin{lemma} \label{perp}
Let $\eps \in (0, 1)$  and let a subspace $G$ of $X$ be $\eps$-orthogonal to a subspace $E$. Then $E$ is $\frac{1}{2-\eps}$-orthogonal to $G$. That is, for all $ e\in E, g\in G$
$$
\|e+g\| \geq \left(1-\frac{1}{2-\eps}\right)\|g\|.
$$
\end{lemma}
\noindent\emph{Proof}.
Since $G$ is $\eps$-orthogonal to $E$, $\|e+g\| \geq (1 - \eps) \|e\|$. On the other hand, $\|e+g\| \geq \|g\|-\|e\|$. Multiply the second inequality by $(1-\eps)$ and add it to the first:
$$
(2-\eps) \|e+g\| \geq (1-\eps) \|g\|.
$$
Then 
$$
\|e+g\| \geq \frac{1-\eps}{2-\eps} \cdot \|g\| = \left(1- \frac{1}{2-\eps}\right)\|g\|. \eqno{\Box}
$$

\begin{lemma} \label{lem-chast-teor}
Let $X$ be a Banach space without a non-trivial infratype. Then for every  $\eps \in (0, 1)$ there are sequences $\{E_k\}_{k\in \mathbb{N}}$ and $\{G_k\}_{k\in \N\cup\{0\}}$ of subspaces of $X$ such that for all $k\in \mathbb{N}$
 
\begin{enumerate}[{$($i$ )$}] 
\item $\dim E_k = 2^{k-1}$, $\codim G_k < \infty$;
\item $G_0 \supset G_1 \supset G_2 \supset \ldots$;
\item $E_k \subset G_{k-1}$; 
\item $d(E_k, \ell_1^{(2^{k-1})}) < 2$;
\item $G_k$ is $\eps$-orthogonal to $E_1+E_2+\ldots + E_k$.
\end{enumerate}
\end{lemma}
\begin{proof}
Since $X$ does not have an infratype, it is not B-convex. Hence \cite[Lemma 7.3.1]{Kadets}, none of finite-codimensional subspaces of $X$ is B-convex. 
 We are going to construct this sequences by induction.  Set $G_0 = X$.

\textbf{The base of induction: $ k = 1$}.  Choose a one-dimensional subspace $E_1 \subset X$. Then automatically $d(E_1, \ell_1^{(1)}) = 1  < 2$. By Lemma \ref{perp_exists} there exists a subspace $G_1 \subset X$ of finite codimension such that $G_1$ is $\eps$-orthogonal to $E_1$. Thus, the conditions above hold for $k=1$. 
 
\textbf{The inductive step}. 
 Suppose we have already constructed $E_i$ and $G_i$ for $i<k$ such that the conditions above hold.  Since $G_{k-1}$ has finite codimension in $X$, it is not $B$-convex, so $\ell_1 \xrightarrow{f} G_{k-1}$. Therefore there exists $E_k \subset G_{k-1}$ with $\dim E_k = 2^{k-1}$ and $d(E_k, \ell_1^{(2^{k-1})}) < 2$.
Then pick a subspace $G_k$ of finite codimension in $G_{k-1}$ such that it is $\eps$-orthogonal to $E_1+E_2+\ldots + E_k$ (it is possible because of Lemma \ref{perp_exists}). 
Also, since $\codim_X G_{k-1}<\infty$, we have $\codim_X G_k < \infty$. It can be easily seen that the constructed elements of the sequences satisfy the conditions above.
\end{proof}

\begin{theorem} \label{thm-without-infratype}
 Let $X$ be a Banach space without a non-trivial infratype. Then there exists a multifunction $F: [0, 1] \to \bb X$ such that $\conv (F)$ is integrable, but $F$ is not.
\end{theorem}
\begin{proof}
Let us apply the previous lemma with $\eps = \frac{1}{2}$ in order to get the corresponding subspaces $\{E_k\}_{k\in \mathbb{N}}$ and $\{G_k\}_{k\in \N\cup\{0\}}$ that satisfy the properties (i)-(v).
Remark that, by $\eps$-orthogonality, $E_k \cap E_j = \{0\}$ for $k \neq j$.

Since $d(\ell_1^{(2^{k-1})}, E_k) < 2$, there exists a linear operator $T_k: \ell_1^{(2^{k-1})} \to E_k$ such that $\|T_k\| \cdot \|T_k^{-1}\| <2$. Replacing $T_k$ with $\lambda \cdot T_k$, we may assume that $\|T_k\| = 1$. Then $\|T_k^{-1}\| < 2$.\newline
 Let $\{\hat{e}_j\}$ be the standard basis of $\ell_1$, then $\{\hat{e}_j\}_{j=1}^{2^{k-1}}$ is a basis of $\ell_1^{(2^{k-1})}$. We can write $T_k\hat{e}_j$ as  $T_k\hat{e}_j = \tau_{2^{k-1}-1 + j} e_{2^{k-1} - 1+j}$, where $\|e_{2^{k-1}-1 + j}\| = 1$. Since $\|T_k\| = 1$ and $\|e_j\| = 1$, $|\tau_{2^{n-1} - 1+j}|\leq 1$. $T_k$ is invertible, so $\{e_j\}_{j=2^{k-1}}^{2^k-1}$ forms a normalized basis of $E_k$.

Now  define the promised function $F: [0, 1] \to \bc X$ to be $F(t) = \{e_j\}_{j\in \mathbb{N}}$ for all $t \in [0, 1]$. Denote $A = \cconv(\{e_j\}_{j\in \mathbb{N}})$. Since $(\conv F)(t) = \conv (\{e_j\}_{j\in \mathbb{N}})$ is convex, Proposition \ref{prop-const-int1} says that $\conv F$ is Riemann integrable with $\integral{(\conv F)(t)} = A$. Therefore, in order to demonstrate that $F$ is not integrable it is sufficient to show that the integral sums of $F$ do not converge to $A$ (Corollary \ref{cor25u}). Consider the tagged partition $T_n$ of the segment $[0,1]$ into $2^{n-1} - 1$ equal parts and denote $S_n$ the corresponding integral sum:
\begin{equation} \label{eqqS_N}
 S_n = S(F, T_n) = \sum \limits_{l = 1}^{2^{n-1}-1}\frac{1}{2^{n-1} - 1} \{e_j\}_{j\in \mathbb{N}} = \left\{\sum \limits_{l = 1}^{2^{n-1}-1}\frac{1}{2^{n-1} - 1} e_{j_l}: j_l \in \N\right\}.
\end{equation}
For each $n\in \mathbb{N}$ choose $y_n = \frac{1}{2^{n}-1}\sum \limits_{j=2^n}^{2^{n+1}-1} e_j \in E_{n+1} \in A$. Then 
\begin{equation} \label{eqqS_N2}
\hausd (S_n, A) \geq \rho(S_n, y_n) = \inf \limits_{s \in S_n} \|y_n-s\|.
\end{equation}
Take an arbitrary $s \in S_n$. According to \eqref{eqqS_N} $s$ is of the form
$s = \sum \limits_{j=1}^\infty t_j e_j$, where $\sum \limits_{j=1}^\infty t_j = 1$ and at most $2^{n-1} - 1$ of $t_j$ are non-zero.
Denote $x = y_n - s = \sum \limits_{j=1}^\infty c_j e_j$. Then  $s$ has at most $2^{n-1}-1$ non-zero coordinates and $y_n$ has exactly $2^{n}-1$  equal to $\frac{1}{2^{n}-1}$, so at least $2^{n-1}$ coordinates of $x$ with numbers between $2^n$-th and $2^{n+1}-1$ are equal to $\frac{1}{2^{n}-1}$.
Denoting $x_k = \sum \limits_{j=2^{k-1}}^{2^k - 1} c_j e_j \in E_k$ for $k\leq n$ and $x_{\infty} = \sum \limits_{j=2^{n+1}}^\infty c_j e_j \in E$ we get $x = \sum \limits_{k=1}^{n+1} x_k + x_{\infty}$. Note that $x_k \in G_{k-1}, x_\infty \in G_{n+1}$. Now we are ready to bound $\|x\|$ from below.
\begin{equation} \label{eqqS_N3}
 \|x\| = \left\|\sum \limits_{k=1}^{n+1} x_k + x_\infty \right\| \geq (1-\eps)  \left\|\sum \limits_{k=1}^{n+1} x_k \right\| = (1-\eps) \left\|\sum \limits_{k=1}^{n} x_k + x_{n+1}\right\|.
\end{equation}
(We used the $\eps$-orthogonality of  $G_{n+1}$ to  $E_1+E_2+\ldots E_{n+1}$ and the fact that $x_\infty \in G_{n+1}$ and $\sum \limits_{k=1}^{n+1} x_k \in E_1+E_2+\ldots + E_{n+1}$).
Remark that $\sum \limits_{k=1}^{n} x_k \in E_1+E_2+ \ldots + E_n$, and $x_{n+1} \in G_{n}$, which is $\eps$-orthogonal to $E_1+E_2+\ldots + E_{n}$, so we may apply Lemma \ref{perp}:
\begin{equation*}
\left\|\sum \limits_{k=1}^{n} x_k + x_{n+1}\right\| \ge   \frac{1 - \eps}{2-\eps} \|x_{n+1}\| = \frac{1}{3}\|x_{n+1}\|,
\end{equation*}
which together with \eqref{eqqS_N3} gives the estimate
\begin{equation} \label{eq5}
 \|x\| \ge \frac{1}{6}\|x_{n+1}\|.
\end{equation}
It remains to bound from below the norm of $x_{n+1} = \sum \limits_{j=2^{n}}^{2^{n+1}-1} c_j e_j\in E_{n+1}$.  
 \begin{align*}
 \|x_{n+1}\| &\geq \frac{1}{\|T_{n+1}^{-1}\|} \cdot \|T_{n+1}^{-1} x_{n+1}\| > 
 \frac{1}{2} \cdot \left\|T_{n+1}^{-1} \left(\sum \limits_{j=2^{n}}^{2^{n +1}-1} c_j e_j\right)\right\| \\ &= \frac{1}{2} \left\|\sum \limits_{j=2^{n}}^{2^{n +1}-1} c_j \cdot T_k^{-1} e_j \right\| = \frac{1}{2} \left\|\sum \limits_{j=2^{n}}^{2^{n +1}-1} c_j \cdot \frac{\hat{e}_{j-2^{n-1}+1}}{\tau_j} \right\| = \frac{1}{2}  \sum \limits_{j=2^{n}}^{2^{n +1}-1} |c_j| \cdot \frac{1}{|\tau_j|}.
 \end{align*}
Plugging in the result into $\eqref{eq5}$ and taking into account that $|\tau_j| \leq 1$, we obtain:
 $$\|x\| \geq \frac{1}{6} \cdot \|x_{n+1}\| \geq \frac{1}{6} \cdot \frac{1}{2} \cdot \sum \limits_{k=2^{n}}^{2^{n+1}-1} |c_k| = \frac{1}{12} \cdot \sum\limits_{k=2^{n}}^{2^{n+1}-1} |c_k| $$
 We know that at least $2^{n-1}$ of $c_j$ are equal to $\frac{1}{2^{n}-1}$. Replace the rest with $0$:
 $$ \|x\| \geq \frac{1}{12} \cdot 2^{n-1} \cdot \frac{1}{2^n-1} \geq \frac{1}{12} \cdot \frac{2^{n-1}}{2^n} = \frac{1}{24}$$
Thus $\hausd(S_n,A)$ is bounded below by $\frac{1}{24}$ and therefore $S_n$ do not converge to $A$.
\end{proof}

\section{The compact case} \label{sect4}
\begin{definition}
 A Banach space $X$ is said to have the \textit{approximation property}, if for every compact subset $A \subset X$ and every $\eps > 0$ there exists a finite-rank operator $P: X \to X$ such that $\|Px - x\| < \eps$ for every $x\in A$.
\end{definition}
We refer to \cite[Chapter 4]{Ryan} for basic examples of spaces with  the approximation property.

\begin{prop}
Let $X$ be a Banach space that satisfies the approximation property.
Let $F: [0,1]\to \bb X$ be a multifunction such that $\conv F$  is Riemann integrable with compact integral (that is, $A = \int\limits_0^1 \conv F(t)\dt$  is compact). Then $F$ is also Riemann integrable. 
\end{prop}
\begin{proof}
Choose an $\eps>0$. 
Since $A$ is compact, there exists a finite rank operator $P$ such that $\|x - P x\| < \eps$ for all $x\in A$. Since $A = \integral{\conv F(t)}$, there is $\delta_1 > 0$ such that for all tagged partitions $T$ of $[0, 1]$, with $d(T) < \delta_1$ the inequality $\hausd(S(\conv F, T), A) < \eps$ holds true. 
Combining Theorem \ref{conv_integr}, Proposition \ref{prop21++} and convexity of $A$ we obtain that multifunction $\conv (P \circ F)$ is integrable and 
$$
\int\limits_0^1 \conv (P \circ F)(t) \dt =  \cconv(\overline{P(A)}) = \overline{P(A)}.
$$
 Now, using Theorem \ref{thm-without-infratype} in the finite-dimensional space $P(X)$, we obtain that $P \circ F$ is also integrable and $\int\limits_0^1 P(F(t)) \dt = \overline{P(A)}$. 
 Therefore there exists $\delta_2 > 0$ such that for all tagged partitions $T$ of $[0, 1]$, with $d(T) < \delta_2$ the inequality $\hausd(S(P \circ F, T), P(A)) < \eps$ holds true.
 
Denote $Q = 1 - P$. By the same reason as above, one can find a constant $\delta_3 > 0$ such that $\hausd(S(Q(\conv F), T), Q(A)) < \eps$ for every tagged partition $T$ of $[0, 1]$, with $d(T) < \delta_3$.

Let $\delta = \min\{\delta_1, \delta_2, \delta_3\}$. 
Fix tagged partition $T$ with $d(T)<\delta$ and denote $S = S(F,T)$ and $\widetilde{S} = S(\conv F, T) = \conv(S)$. By the choice of $\delta$, we have that
$$\hausd(\conv (S), A) < \eps, \hausd(P (S), P (A)) < \eps, \hausd(Q (\widetilde{S}), Q(A)) = \hausd(\conv (Q (S)),Q(A)) < \eps
$$
at once. Also we know that $\|x -P x\| = \|Q x\| < \eps$  for all  $x\in A$, so $Q(A) \subset  \eps B_X$.
Since $\hausd(\conv (Q(S)), Q(A)) < \eps$, for every $x \in \conv (Q (S))$ (and in particular for every $x \in Q(S)$) there exists $v \in Q(A)$ with $\|x - v\| < \eps$. Therefore $\|x\| \leq \|x - v\| + \|v\| < 2\eps$. Consequently, for $x \in Q(S)$, $y \in Q(A)$
$$
 \|x - y\| \leq \|x\| + \|y\| < 3\eps.
$$
Now let us demonstrate that $\hausd (S, A) < 4\eps$.

Indeed, for every $s\in S$  we can approximate up to $\eps$ the element $P s \in P(S)$ by an element of $P(A)$. In other words, there is $a_0 \in A$ such that $\|P s - P a_0\| < \eps$. Then 
$$
\|s - a_0\| = \|(P s + Q s) - (P a_0 + Q a_0)\| \leq \|P s - P a_0\| + \|Q s - Q a_0\| < \eps+ 3 \eps< 4\eps.
$$
The same way, approximating for arbitrary $a \in A$ the corresponding $Pa \in P(A)$ by an element of $P(S)$ we obtain $s_0 \in S$ with $\|P s_0 - P a\| < \eps$. Then 
$$
\|s_0 - a\| = \|(P s_0 + Q s_0) - (Pa + Q a)\| \leq \|P s_0 - Pa\| + \|Q s_0 - Q a\| < \eps+ 3 \eps< 4\eps.
$$ 
Summarizing, for every $\eps> 0$ we have found a $\delta>0$ such that $\hausd(S(F,T), A) < 4\eps$ for all tagged partitions $T$ with $d(T)<\delta$. This means that $\int\limits_0^1 F(t) \dt = A$.
\end{proof}

\begin{theorem}
Let $X$ be a Banach space. Let $F: [0,1]\to \bb X$ be a multifunction such that its convex hull is Riemann integrable and $A = \int\limits_0^1 \conv F(t)\dt$ is compact. Then the multifunction $F$ is also Riemann integrable (and, thanks to Corollary \ref{cor25u},  $\int\limits_0^1 F(t)\dt = A$). 
\end{theorem}
\begin{proof}
Consider $K = \overline{B(X^*)}$ in the weak$^*$ topology. By the Banach-Alaoglu theorem, $K$ is compact. Define a map $U: X \to C(K)$: 
$$
(U(x))(x^*) = x^*(x).
$$ 
$U$ is linear, $\|Ux\| = \sup_{x \in K} |x^*(x)| = \sup_{x \in B_X} |x^*(x)| = \|x\|$ for all $x \in X$, therefore $U$ is an isometric embedding.
According to \cite[example~4.2]{Ryan}, $C(K)$ possesses the approximation property. Thus, by the previous proposition, our Theorem is true in $C(K)$. Consequently, the Theorem is true for every closed linear subspace of  $C(K)$, in particular, it is true in $U(X)$, which is just an isometric copy of $X$.
\end{proof}

Since the closed limit of a sequence of precompacts is compact, we arrive at the following fact:

\begin{corollary}
Let $X$ be a Banach space. Let $F: [0,1] \to \bb{X}$ be a multifunction with precompact values such that its convex hull is Riemann integrable. Then the multifunction $F$ is also Riemann integrable.
\end{corollary}

\end{document}